\documentclass{gtmon_a}
\pdfoutput=1
\usepackage{array,enumerate}
\usepackage[all]{xy}
\input{diagxy}
\let\to\rightarrow
\let\barrsquare\square
\let\square\undefined

\proceedingstitle{Proceedings of the School and Conference in Algebraic
Topology (The Vietnam National University, Hanoi, 9-20 August 2004)}
\conferencestart{9 August 2004}
\conferenceend{20 August 2004}
\conferencename{School and Conference in Algebraic Topology}
\conferencelocation{Vietnam National University, Hanoi, Vietnam}

\editor{John Hubbuck}
\givenname{John}
\surname{Hubbuck}

\editor{Nguy\~\ecircumflex{}n H V H\uhorn{}ng}
\givenname{Nguy\~\ecircumflex{}n}
\surname{Hu'ng}

\editor{Lionel Schwartz}
\givenname{Lionel}
\surname{Schwartz}

\title{Modular invariants detecting the cohomology of\\$BF_4$ at the prime $3$}

\author{Carles Broto}
\givenname{Carles}
\surname{Broto}
\address{Departament de Matem\`atiques\\
Universitat Aut\`onoma de Barcelona\\\newline
08193 Bellaterra\\
Spain}
\email{broto@mat.uab.es}
\urladdr{}

\volumenumber{11}
\issuenumber{}
\publicationyear{2007}
\papernumber{1}
\startpage{1}
\endpage{16}

\doi{}
\MR{}
\Zbl{}

\arxivreference{}


\received{22 April 2005}
\revised{}
\accepted{4 October 2005}
\published{14 November 2007}
\publishedonline{14 November 2007}
\proposed{}
\seconded{}
\corresponding{}
\editor{}
\version{}

\makeatletter
\def\cnewtheorem#1[#2]#3{\newtheorem{#1}{#3}[section]
\expandafter\let\csname c@#1\endcsname\c@thm}

\let\xysavmatrix\xymatrix
\def\xymatrix{\disablesubscriptcorrection\xysavmatrix}
\AtBeginDocument{\let\bar\wbar\let\tilde\widetilde\let\hat\what}

\makeop{Spin}

\newtheorem{thm}{Theorem}[section]
\cnewtheorem{cor}[thm]{Corollary}
\cnewtheorem{lem}[thm]{Lemma}
\cnewtheorem{prop}[thm]{Proposition}
\cnewtheorem{que}[thm]{Question}
\theoremstyle{definition}
\cnewtheorem{defn}[thm]{Definition}
\theoremstyle{remark}
\cnewtheorem{rem}[thm]{Remark}
\cnewtheorem{example}[thm]{Example}
\numberwithin{equation}{section}

\makeatother  

\newcommand{\To}{\longrightarrow}

\newcommand{\conj}[2]{\bigl\{\,#1\,\big|\,#2\,\bigr\}}
\newcommand{\Hom}{\operatorname{Hom}\nolimits}

\newcommand{\im}{\operatorname{Im}\nolimits}
\newcommand{\coker}{\operatorname{Coker}\nolimits}
\newcommand{\res}{\operatorname{res}\nolimits}
\newcommand{\gcdd}{\operatorname{gcd\,}\nolimits}
\newcommand{\fp}{\mathbb F_p}
\newcommand{\ff}{\mathbb F}
\newcommand{\zp}{\mathbb Z/p}
\newcommand{\pp}{\mathcal P}
\newcommand{\kk}{\mathcal K}
\newcommand{\uu}{\mathcal U}
\newcommand{\nn}{\mathcal N}
\newcommand{\cc}{\mathcal C}

\newcommand{\F}{{\mathbb F}}
\newcommand{\pv}{P(V^*)}
\newcommand{\ev}{E(dV^*)}
\newcommand{\kv}{K(V^*)}
\newcommand{\fpv}{FP(V^*)}
\newcommand{\fkv}{FK(V^*)}
\newcommand{\glv}{GL(V)}
\newcommand{\gll}{GL}
\newcommand{\fpvg}{FP(V^*)^G}

\newcommand{\pvg}{P(V^*)^G}
\newcommand{\kvg}{K(V^*)^G}
\newcommand{\rhoun}{\rho_1,\ldots, \rho_n}
\newcommand{\drhoun}{d\rho_1,\ldots, d\rho_n}
\newcommand{\drho}{d\rho_1\cdots d\rho_n}
\newcommand{\prho}{P [\rho_1, \ldots, \rho_n]}
\newcommand{\pu}{P [u_1, u_2, u_3]}
\newcommand{\evv}{E [v_1, v_2, v_3]}
\newcommand{\edx}{E[dx_1, \ldots, dx_n]}
\newcommand{\edrho}{E[d\rho_1, \ldots, d\rho_n]}
\newcommand{\pxun}{P[x_1, \ldots, x_n]}

\begin{document}

\begin{asciiabstract}
Attributed to J F Adams is the conjecture that, at odd primes, 
the mod-p cohomology ring of the classifying space of a connected compact 
Lie group is detected by its elementary abelian p-subgroups.
In this note we rely on Toda's calculation of H^*(BF_4;F_3)
in order to show that the conjecture holds in case of the exceptional
Lie group F_4.  To this aim we use invariant theory in order to identify
parts of H^*(BF_4;F_3) with invariant subrings in the cohomology of
elementary abelian 3-subgroups of F_4.  These subgroups themselves
are identified via the Steenrod algebra action on H^*(BF_4;F_3).
\end{asciiabstract}

\begin{htmlabstract}
Attributed to J&nbsp;F&nbsp;Adams is the conjecture that, at odd
primes, the mod&ndash;p cohomology ring of the classifying space of
a connected compact Lie group is detected by its elementary abelian
p&ndash;subgroups.  In this note we rely on Toda's calculation of
H<sup>*</sup>(BF<sub>4</sub>;<b>F</b><sub>3</sub>) in order to show that
the conjecture holds in case of the exceptional Lie group F<sub>4</sub>.
To this aim we use invariant theory in order to identify parts of
H<sup>*</sup>(BF<sub>4</sub>;<b>F</b><sub>3</sub>) with invariant
subrings in the cohomology of elementary abelian 3&ndash;subgroups of
F<sub>4</sub>.  These subgroups themselves are identified via the Steenrod
algebra action on H<sup>*</sup>(BF<sub>4</sub>;<b>F</b><sub>3</sub>).
\end{htmlabstract}

\begin{abstract}
Attributed to J\,F~Adams is the conjecture that, at odd primes, 
the mod--$p$ cohomology ring of the classifying space of a connected compact 
Lie group is detected by its elementary abelian $p$--subgroups.
In this note we rely on Toda's calculation of $H^*(BF_4;\mathbb{F}_3)$
in order to show that the conjecture holds in case of the exceptional
Lie group $F_4$.  To this aim we use invariant theory in order to identify
parts of $H^*(BF_4;\mathbb{F}_3)$ with invariant subrings in the cohomology of
elementary abelian $3$--subgroups of $F_4$.  These subgroups themselves
are identified via the Steenrod algebra action on $H^*(BF_4;\mathbb{F}_3)$.
\end{abstract}

\maketitle

\section{Introduction}

It has been known since the work of Borel \cite{Borel1,Borel2} that the
rational cohomology of the classifying space of a compact and connected
Lie group $G$ is detected on its maximal torus $T_G$, and can actually
be identified with the subalgebra of elements in the cohomology of
the classifying space of $T_G$ that are fixed by the natural action
of the Weyl group $W_G$; that is, we can identify $H^*(BG;\Q)\cong
H^*(BT_G;\Q)^{W_G}$. Similar identifications hold for cohomology
with coefficients in fields of prime charateristic as soon as this
charateristic does not divide the order of the Weyl group.

A quick look at the mod $p$ cohomology of classifying spaces of compact
connected Lie groups at torsion primes (cf Mimura and Toda \cite{MimTod})
shows that restrictions to  maximal tori usually have big kernels.
In particular all odd degree elements can only be mapped trivially by
the restriction to the maximal torus. We are then led to consider the
restriction to elementary abelian subgroups. At odd primes, there
is always a maximal one that consists of all elements of $p$--power
order in the maximal torus, but in presence of torsion there are also
elementary abelian subgroups which are non-toral; that is, not conjugate
to a subgroup of the maximal torus. If $\mathcal{E}_p(G)$ is a set of
representatives of all conjugacy classes of maximal elementary abelian
$p$--subgroups, then the kernel of the restriction map
$$q_G\colon H^*(BG;\F_p)\longrightarrow
  \prod_{E\in\mathcal{E}_p(G)}H^*(BE;\F_p)$$
is nilpotent, according to Quillen \cite{Quillen}, for any compact Lie
group $G$ and any prime $p$.  Adams conjectured  that $q_G$ is actually
a monomorphism if $G$ is compact and connected and $p$ is an odd prime.

In this note we rely on Toda's calculation \cite{Toda} of $H^*(BF_4;\F_3)$
to show that $q_G$ is a monomorphism in this case. Kono and Yagita
\cite{KonoYagita}  proved that $q_G$ is a monomorphism for $G=PU(3)$
at the prime $3$. This has been recently generalized  by Vavpeti\v{c}
and Viruel \cite{VavpeticViruel} to $G=PU(p)$ at the prime $p$, for $p$
odd. Mimura, Sambe, Tezuka, and Toda \cite{MSTT} have also obtained that
the conjecture is true for $G=E_6$ at $p=3$.

If $W_G(E)$ denotes the group of automorphisms of the elementary
abelian subgroup $E$ of $G$ which are induced by conjugation
in $G$, the restriction map has image in the invariant subring
$H^*(BE;\F_p)^{W_G(E)}$. In section two we present the relevant
invariant theory in  order to have a description as algebras over
the Steenrod algebra of these invariant rings for the elementary
abelian $3$--subgroups of $F_4$ that are involved in our calculations.
These subgroups were  identified by Rector \cite[Section~7]{Rector}
by arguments based on work of Toda, and confirmed by Adams using geometric
arguments. Taking Rector's calculations as starting point and comparing
the Steenrod algebra action on $H^*(BF_4;\F_3)$ and on the invariant
subrings $H^*(BT_{F_4};\F_3)^{W_{F_4}}$ and $H^*(BE;\F_3)^{W_{F_4}(E)}$,
we obtain a precise description of $q_{F_4}$ at the prime $3$ in section
three, and,  in particular, that it is injective.

Part of the results presented here were announced in \cite{Bro1} but
details remained unpublished at that time. Now, during the Conference in
Algebraic Topology held in Hanoi in honor of Hu\`ynh Mui's 60th birthday,
in August, 2004, I saw a renewed interest in the subject by M~Kameko,
M~Mimura, and A~Viruel, among others, which prompted me to submit this
note to the conference proceedings.

The author is partially supported by FEDER/MEC grant MTM2004--06686.

\section{Modular invariants in $P(V)\otimes E(V)$}

Let $\F_p$ be the field of $p$ elements, for an odd prime $p$, and $V$ a
$\F_p$--vector space of dimension $n$. We denote by $\pv$ the
symmetric algebra on the dual vector space $V^*$. If $d\colon V^* \to
dV^*$ is an isomorphism of $\F_p$--vector spaces and $\ev$ is the
exterior algebra on $dV^*$, $d$ extends uniquely to a derivation of the  
 algebra
\begin{equation*}
\kv=\pv \otimes \ev=\pxun\otimes \edx.
\end{equation*}
Let $G$ be a finite subgroup of $GL(V)$, so that $G$ acts on
$\pv$ via the transpose representation. This action can be extended 
to a differential algebra action on $\kv$
in a unique way. We are interested in cases in which the fixed
subalgebra is again a polynomial algebra on $n$ generators
\begin{equation*}
\pvg=\prho\,.
\end{equation*}
In this case, according to a theorem of Serre, $G$ is a pseudoreflection
group, i.e., $G$ is generated by elements that fix a codimension one
subspace of $V$ (see Bourbaki \cite{Bourbaki}, Benson \cite{Benson} and
Smith \cite{LSmith}).  If the order of $G$ is not divisible by $p$, then
$\pvg$ is a polynomial algebra if and only if $G$ is a pseudoreflection
group and in this case the ring of invariants of $\kv$ is
$$\kvg = \prho \otimes \edrho,$$
(see Solomon \cite{Solomon}, Benson \cite{Benson} and Smith \cite{LSmith}).

In this section we will discuss the case in which $\pvg$ is a polynomial
algebra but $p$ divides the order of $G$.  First, we will show a necessary
and sufficient condition under which the above formula holds. In the
case that this condition is not satisfied we show how to construct
new invariants in $\kv$. In the cases that we have checked, these new
invariants provide a complete system of generators for $\kvg$. For
$G=GL(V)$ or $G=SL(V)$ they can be compared with the system obtained by
Mui \cite{Mui}.

The modules of relative invariants will play a fundamental role in our
discussion. Recall that for a linear character of $G$, $\chi\colon
G\to \F_p^*$, with $\chi(gh)=\chi(g)\chi(h)$, the $\pvg$--module of
$\chi$--relative invariants is defined
$$\pvg_\chi = \conj{q\in\pv}{g\cdot q = \chi(g)q, \ \forall g\in G}\,.$$
If $G$ is a pseudoreflection group, it turns out that this is a free
$\pvg$--module of rank one, $\pvg_\chi= f_\chi\pvg$, for an element
$f_\chi\in\pv$ that can be written in a unique way, up to scalar
multiplication, as a product of forms in $V^*$ (see Stanley \cite{Stanley}
and Broto--Smith--Stong \cite{BSS}).

If we write $d\rho_i =\sum_{i=1}^n a_{ij}dx_j$, the jacobian,
$J=\det (a_{ij})_{i,j}$ is a non-trivial element of $\pv$ (see
Wilkerson \cite{Wilkerson}).
Moreover, it is a $\det^{-1}$--relative invariant:
$$J\in \pvg_{\det^{-1}}=f_{\det^{-1}}\cdot \pvg.$$
It follows that $f_{\det^{-1}}$ always divides $J$.

\begin{thm}[Broto \cite{Bro2}]
Let $\F_p$ be the field of $p$ elements, where $p$ is an odd prime, and $V$ a
$\F_p$--vector space of dimension $n$. Assume that $G$ is a finite
subgroup of $\glv$ such that $\pvg=\prho$, then
\begin{equation*}
\kvg =\prho\otimes \edrho
\end{equation*}
if and only if $J=f_{\det^{-1}}$ (up to an invertible of $\F_p$).
\end{thm}
\begin{proof}
As in the characteristic zero case, the fact $J\neq 0$, implies
that the morphism
\begin{equation*}
\prho\otimes \edrho \to \kv
\end{equation*}
is always injective. We need to determine the cases in which all
elements of $\kvg$ belong to the image; in other words, are expressible
as algebraic combination of $\rhoun,\drhoun$.

If $I=(i_1, \ldots, i_k)$ is an ordered sequence of integers:
$0<i_1<\cdots < i_k \leq n$, we denote $d\rho_I=d\rho_{i_1} d\rho_{i_2}
\cdots d\rho_{i_k}$ or $dx_I=dx_{i_1}dx_{i_2}\cdots dx_{i_k}$. With
this notation, $\kv$ is a free $\pv$--module generated by $\{dx_I\}_I$,
and so, if $\fpv$ is the field of fractions of $\pv$ and
$$\fkv=\fpv\otimes_{\pv} \kv,$$
then $\fkv$ is a $\fpv$--vector space with base $\{dx_I\}_I$, and
$\{d\rho_I\}_I$ form a base, too.

Assume first that $J=f_{\det^{-1}}$ up to an scalar.  Choose an arbitrary
element $w\in \kvg$. It may be written  as a linear combination $w=\sum_I
w_Id\rho_I$, with $w_I \in \fpvg$. We will show that each $w_I$  lies
in $P(V)^G$ so that $w\in \prho \otimes \edrho$.

Choose a sequence $I_0$ of minimal
length such that $w_{I_0}\neq 0$ and  let $I'_0$ be the complementary
sequence, so that the expression
$$wd\rho_{I'_0} = w_{I_0}d\rho_{I_0}d\rho_{I'_0}
  = \pm w_{I_0}d\rho_1\cdots d\rho_n
  = \pm w_{I_0} Jdx_1 \cdots dx_n$$
is still an element of $\kvg$. But $dx_1 \cdots dx_n$ is invariant
relative to $det$, hence $w_{I_0}J \in\pvg_{\det^{-1}}
=f_{\det^{-1}}\pvg$, and so therefore $w_{I_0}\in \pvg$. We repeat the
process with $w-w_{I_0}d\rho_{I_0}$ and we obtain inductively that all
coefficients $w_I\in\pvg$.

Assume now that $J=\iota\, f_{\det^{-1}}$ for a positive degree polynomial
$\iota \in \pvg$, then
\begin{equation*}
w=\dfrac{\drho}{\iota}=f_{\det^{-1}}\, dx_1\cdots dx_n
\end{equation*}
is an element in $K(V)^G$ but it does not belong to the subalgebra $\prho
\otimes \edrho$.
\end{proof}

We have seen that whenever the jacobian $J$ is different from
$f_{\det^{-1}}$ in an essential way,
 we obtain a new invariant that does not belong to
$\prho \otimes \edrho$, by dividing $\drho$ by an invariant factor of
$P(V^*)$. A similar argument applies  to any $d\rho_I$; that is, we can
divide each $d\rho_I$ by its \emph{maximal invariant factor}, $B_I\in
\pvg$, in order to obtain a new invariant: $M_I=\frac{1}{B_I}d\rho_I
\in K(V)^G$.

More precisely, fix a sequence $I$, and write 
\begin{equation*}
d\rho_I=\sum_J a_J (I)dx_J, \quad a_J\in \pv
\end{equation*}
where all sequences $J$ have the same length as $I$. We define $A_I= \gcdd
(a_J(I))$ so that $d\rho_I=A_I\sum_J b_J(I)dx_J$ and the coefficients
$b_J(I)\in\pv$ have no common factor.

It turns out that $A_I$ is relative invariant to a certain linear character
$\chi_I$ of $G$. In fact, for any $g\in G$
$$\sum_J a_J(I)d x_J  = d\rho_I   
 = g(d\rho_I) 
 = g\Bigl(A_I\sum_J b_J (I)dx_J\Bigr)
 = g(A_I)\sum_Jb'_J(I)dx_J$$
hence $g(A_I)$ divides each $a_J(I)$ and so therefore, $g(A_I)=\chi_I
(g)A_I$ for some element $\chi_I (g)\in \F_p^*$. This defines the
character $\chi_I$, and shows that $A_I\in \pvg_{\chi_I}$.

Since $\pvg_{\chi_I}=f_{\chi_I}\pvg$ for certain class $f_{\chi_I} \in
\pvg_{\chi_I}$, we can define elements $B_I\in \pvg$ by the equation
$A_I=B_I\cdot f_{\chi_I}$, and then $M_I{=}f_{\chi_I}
\sum_Jb_J(I)dx_J\allowbreak\in \kvg$ gives the factorization
\begin{equation*}
d\rho_I=B_I\cdot M_I
\end{equation*}
with $M_I\in \kvg$ and $B_I\in \pvg$. Notice also that by construction
we obtain relations
$$M_IM_J = \begin{cases}
  q_{I,J}M_{I\cup J} & \text{for some } q_{I,J}\in\pvg,
    \text{ if } J\cap I=\emptyset, \\
  0 & \text{if } J\cap I\neq\emptyset.
  \end{cases}$$
It seems reasonable to ask whether or not we have obtained a complete
system of generators and relations for $\kvg$.
\begin{que}
Is $\kvg$ a free $\pvg$--module generated by $\{M_I\}_I$, for every
group $G\leq GL(V)$ for which $\pvg$ is a polynomial algebra?
\end{que}

We have a positive answer in the cases which are involved in the 
mod $3$ cohomology of $BF_4$.

\begin{example}\label{exemple1} 
Let $\fp$ be the field of $p$ elements, $p$ an odd prime, and assume
 $\kv=P[x_1,x_2]\otimes E[dx_1, dx_2]$, and
$G=\gll_2(\fp)$. The Dickson invariants are described, in terms of determinants of two by two matrices as
$$L_2=\left | \begin{matrix} x_1 & x_2 \\ x_1^p & x_2^p
\end{matrix}\right | \,, \qquad
Q_{2,1} = \left | \begin{matrix} x_1 & x_2 \\ x_1^p & x_2^p
\end{matrix} \right |^{-1}\cdot
\biggl | \begin{matrix} x_1 & x_2 \\ x_1^{p^2} & x_2^{p^2}
\end{matrix} \biggr |$$
and we have $P[x_1,x_2]^{\gll_2(\fp)}=P[L_2^{p-1}, Q_{2,1}]$ (see Dickson
\cite{Dickson}).
One then obtains
$$dL_2^{p-1}=-L_2^{p-2}(x_2^pdx_1-x_1^pdx_2)\quad\text{and}\quad
  dQ_{2,1}=-L_2^{p-2}(x_2dx_1-x_1dx_2).$$
Since
$L_2^{p-2}=f_{\det^{-1}}$, we have
$$M_1= dL_2^{p-1} \quad\text{and}\quad M_2=dQ_{2,1}$$
with $B_1=B_2=1$.
On the other hand, $dL_1^{p-1}dQ_{2,1}=-L_2^{2p-3}dx_1dx_2$, hence
$$M_{1,2}=-L_2^{p-2}dx_1dx_2$$
with $M_1M_2=L_2^{p-1}M_{1,2}$. According to Mui \cite{Mui}, we know  that
$$\bigl\{L_2^{p-1},Q_{2,1}, M_1, M_2, M_{1,2}\bigr\}$$
is a full system of generators for                     
$(P[x_1, x_2]\otimes E[dx_1, dx_2])^{\gll_2(\fp)}$.
\end{example}

\begin{example}\label{exemple2}
We describe the invariants of
$$H^*((\zp)^3; \fp)=\pu\otimes \evv,$$
$\deg v_i=1$, $u_i=\beta v_i$ by the action of $SL_3 (\fp)$, for an odd
prime $p$.

The Dickson invariants are the determinants
\begin{align*}
L_3& =\left | \begin{matrix} u_1 & u_2 & u_3\\ u_1^p & u_2^p &
u_3^p \\ u^{p^2}_1 & u^{p^2}_2 & u^{p^2}_3
\end{matrix}\right | \,, & \deg L_3&=2\dfrac{p^3-1}{p-1} \\
Q_{3,2}& =\dfrac{1}{L_3}\left | \begin{matrix} u_1 & u_2 & u_3\\
u_1^p & u_2^p & u_3^p \\ u^{p^3}_1 & u^{p^3}_2 & u^{p^3}_3
\end{matrix}\right | \,, & \deg Q_{3,2}&=2(p^3-p^2) \\
Q_{3,1}& =\dfrac{1}{L_3}\left | \begin{matrix} u_1 & u_2 & u_3\\
u_1^{p^2} & u_2^{p^2} & u_3^{p^2} \\ u^{p^3}_1 & u^{p^3}_2 & u^{p^3}_3
\end{matrix}\right | \,, & \deg Q_{3,1}&=2(p^3-p) \\
\end{align*}
The action of the Steenrod algebra on these elements is determined
by
\begin{align*}
\pp^1L_3&=0 & \pp^1 Q_{3,2}&=0 & \pp Q_{3,1}&=L_3^2\\*[2pt]
\pp^pL_3&=0 & \pp^pQ_{3,2}&=Q_{3,1} & \pp^p Q_{3,1}&=0 \\*[2pt]
\pp^{p^2} L_3&=Q_{3,2}L_3 & \pp^{p^2}Q_{3,2}&=-Q^2_{3,2} &
\pp^{p^2}&=-Q_{3,2}Q_{3,1}
\end{align*}
(see Dickson \cite{Dickson} and Wilkerson \cite{Wilk-Dickson}).
Since det is trivial on $SL_3 (\fp)$, $M_{1,2,3}=v_1v_2v_3$ is invariant. 
Steenrod operations can
be used to find new invariants: \setlength{\extrarowheight}{1.7pt}
\begin{align*}
M_{2,3}&=\beta M_{1,2,3}=u_1v_2v_3 - u_2v_1v_3 + u_3v_1v_2,\\
M_{1,3}&=-\pp^1 M_{2,3}=-u_1^pv_2v_3 - u_2^pv_1v_3 + u^p_3v_1v_2,\\
M_{1,2}&=\pp^{p} M_{1,3}=-u_1^{p^2}v_2v_3 + u_2^{p^2}v_1v_3 - u^{p^2}_3v_1v_2,\\
M_3&=\beta M_{1,3}= \biggl | \begin{matrix} u_2 & u_3 \\ u_2^p & u_3^p
\end{matrix}\biggr|v_1 - \left | \begin{matrix} u_1 & u_3 \\ u_1^p & u_3^p
\end{matrix}\right |v_2 + \left | \begin{matrix} u_1 & u_2 \\ u_1^p & u_2^p
\end{matrix}\right |v_3,\\
M_{2}&=\pp^p M_3=\biggl  | \begin{matrix} u_2 & u_3 \\ u_2^{p^2}
& u_3^{p^2}\end{matrix}\biggr  |v_1 - \biggl | \begin{matrix} u_1 &
u_3\\ u_1^{p^2} & u_3^{p^2}\end{matrix}\biggr  |v_2 + \biggl |
\begin{matrix} u_1 & u_2 \\ u_1^{p^2} & u_2^{p^2}
\end{matrix}\biggr  |v_3,\\
M_{1}&=\pp^1 M_{2}= \biggl  | \begin{matrix} u_2^p & u_3^p \\
u_2^{p^2} & u_3^{p^2}\end{matrix}\biggr  |v_1 -\biggl  |
\begin{matrix} u_1^p & u_3^p\\ u_1^{p^2} & u_3^{p^2}\end{matrix}\biggr 
|v_2 + \biggl  | \begin{matrix} u_1^p & u_2^p \\ u_1^{p^2} &
u_2^{p^2} \end{matrix}\biggr |v_3,
\end{align*}
and finally $\beta M_{1}=L_3$. One can check that these are precisely the 
set of invariants described above:
\begin{align*}
M_1&=dL_3, & M_2&=\frac1{L_3}\,dQ_{3,2}, & M_3&=\frac1{L_3}\,dQ_{3,1},\\
M_{1,2}&=\frac1{L_3{}^2}\,dL_3\,dQ_{3,2}, &
M_{1,3}&=\frac1{L_3{}^2}\,dL_3\,dQ_{3,1}, &
M_{2,3}&=\frac1{L_3{}^2}\,dQ_{3,2}\,dQ_{3,1},
\end{align*}
$$\text{and}\quad M_{1,2,3}=-\frac1{L_3{}^4}\,dL_3\,dQ_{3,2}\,dQ_{3,1}.$$
Again in this case, according to Mui \cite{Mui},  
$$\{L_3, Q_{3,1}, Q_{3,2}, M_1, M_2,
M_3, M_{1,2}, M_{1,3}, M_{2,3}, M_{1,2,3}\}$$ 
forms a full system of
generators for $H^*((\Z/p)^3; \fp)^{SL_3 (\fp)}$.
\end{example}

\section{On the cohomology of $BF_4$ at prime 3}
In this section we will show how starting with the computation of $H^*(BF_4;\ff_3)$
by Toda \cite{Toda}, one can obtain that this cohomology ring is detected on elementary 
abelian $3$--subgroups. The argument goes through the description of 
$H^*(BF_4;\ff_3)/\sqrt{0}$ by Rector \cite{Rector}.

For the convenience of the reader we present here Toda's description of 
the cohomology of the classifying space of the exceptional Lie
group $F_4$ at $p=3$. 
The Weyl group of $F_4$ contains the Weyl group of
$\Spin_9$, so that we have:
\begin{equation*}
H^*(BT;\ff_3)^{W_{F_4}}\subset H^*
(BT;\ff_3)^{W_{\Spin_9}}=P[p_1, p_2,p_3,p_4]\,,
\end{equation*}
where $p_i$ are Pontrjagin classes. Toda first computed the invariant ring
$$
H^*(BT;\ff_3)^{W_{F_4}}=P[p_1,\bar{p_2},
\bar{p_5},\bar{p_9},\bar{p_{12}}]/(r_{15})
$$ 
where
\begin{align*}
\bar p_2&=p_2 - p_1^2, \\
\bar p_5&=p_4p_1 + p_3\bar p_2,\\
\bar p_9&=p_3^3 - p_4p_3p_1^2 + p_3^2\bar p_2p_1 - p_4\bar p_2p_1^3,\\
\bar p_{12}&=p_4^3 + p_4^2\bar p_2^2 + p_4\bar p_2^4,\\
r_{15}&=\bar p_5^3 + \bar p_5^2\bar p_2^2p_1 - \bar p_{12}p_1^3
   - \bar p_9\bar p_2^3,
\end{align*}
and obtained elements  $x_4$, $x_8$, $x_{20}$, $x_{36}$, $x_{48}$ in
$H^*(BF_4;\ff_3)$ that restrict to $p_1$, $\bar p_2$,
$\bar p_5$, $\bar p_9$, and $\bar
p_{12}$, respectively in $H^*(BT;\ff_3)^{W(F_4)}$. 

\begin{thm}[Toda \cite{Toda}]\label{TodaThm}
$H^*(BF_4;\ff_3)$ is an algebra generated by
\begin{equation*}
\begin{array}{c}
x_4, \quad x_8, \quad  x_{20}, \quad  x_{36}, \quad  x_{48}, \\
x_9=\beta x_8 ,\quad 
 x_{21}=\beta x_{20}, \quad  x_{25}=\pp^1 x_{21}, \quad  x_{26}=\beta x_{25},
\end{array}
\end{equation*}
with the relations
\begin{align*}
x_9x_4&=x_9x_8=x_9^2=x_{21}x_4=x_{25}x_8 \\
&=x_{21}x_{20}=x_{21}^2 =x_{25}x_{20}=x_{25}^2=0, \\
x_{21}x_8&=x_{25}x_4 =-x_{20}x_9, \\
x_{26}x_4&=-x_{21}x_9,\\
x_{26}x_8&=x_{25}x_9,\\
x_{25}x_{21}&=x_{26}x_{20},\\
x_{20}^3&=x_{48}x_4^3 + x_{36}x_8^3-x_{20}^2x_8^2x_4.
\end{align*}
Furthermore, 
the Steenrod algebra action on $H^*(BF_4;\ff_3)$ is completely
determined by the following relations:
\begin{center}
\begin{tabular}{|>{\scriptsize}c|>{\scriptsize}c|>{\scriptsize}c|
>{\scriptsize}m{3cm}|>{\scriptsize}m{5.15cm}|} \hline
 & \multicolumn{1}{c|}{$\beta$} & \multicolumn{1}{c|}{$\pp^1$} &
 \multicolumn{1}{c|}{$\pp^3$} & \multicolumn{1}{c|}{$\pp^9$} \\
 \hline
 $x_4$ &  & $-x_8+x_4^2$ & & \\
 \hline
 $x_8$ & $x_9$ & $x_8x_4$ & \centering{$x_{20} -x_8^2x_4$} &  \\
 \hline
 $x_9$ & & & \centering{$x_{21}$} & \\
 \hline
  $x_{20}$ & $x_{21}$ & & \centering{$x_{20}\bigl(-x_8+x_4^2\bigr)$} &
  $\bigl(x_{48}+x_{20}^2x_8\bigr)\bigl(-x_8+x_4^2\bigr) \newline
  +x_{36}\bigl(x_{20}+x_8^2x_4\bigr)
  +x_{26}x_{21}x_9$ \\
 \hline
$x_{21}$ & & $x_{25}$ &  & $-x_{48}x_9+x_{36}x_{21}$\\
 \hline
 $x_{25}$ & $x_{26}$ & & &  $x_{36}x_{25}-x_{26}^2x_9$ \\
 \hline
 $x_{26}$ & & & & $x_{36}x_{26}$\\
 \hline
$x_{36}$ & & $-x_{20}^2$ & $x_{48}-x_{36}\bigl(x_8+x_4^2\bigr)x_4$
$+x_{20}^2 \bigl(x_8+x_4^2\bigr)$
 & $-x_{48}x_{20}x_4+x_{48}\bigl(x_8^2+x_4^4\bigr)x_4^2  -x_{36}^2$
   $+x_{36}x_{20}\bigl(x_8+x_4^2\bigr)x_4^2-x_{36}\bigl(x_8^2+x_4^4\bigr)^2x_4$
$+x_{20}^2x_8 \bigl(x_8^3+\bigl(x_8+x_4^2\bigl)^2x_4^2\bigr)$ \\
  \hline
  $x_{48}$ & & $x_{26}^2$ & \centering{$-x_{48}\bigl(x_8+x_4^2\bigr)x_4$} &
$-x_{48}x_{36}+x_{48}x_{20}\bigl(-x_8^2-x_8x_4^2 +x_4^4\bigr)$
   $-x_{48} \bigl(x_8^2+x_4^4\bigr)^2 x_4$ \\
   \hline
\end{tabular}
\end{center}
\end{thm}

The important observation of Rector concerning the cohomology of $BF_4$
at the prime $3$, is that the quotient of $H^*(BF_4;\ff_3)$ by its radical
$\sqrt0$, the ideal of all nilpotent elements, can be better understood
that $H^*(BF_4;\ff_3)$ itself and carries most of its information.

Recall that the radical of an algebra $K$, is defined as
\begin{equation*}
\sqrt{0}=\conj{x\in K}{x^r=0 \text{ for some integer } r}.
\end{equation*}
It follows from \fullref{TodaThm} that the radical of $H^*(BF_4;\ff_3)$
is the ideal generated by $x_9$, $x_{21}$, $x_{25}$ and then $H^* (BF_4;
\ff_3)/\sqrt 0$ is generated by classes
\begin{equation*}
x_4, x_8, x_{20}, x_{26}, x_{36}, x_{48}
\end{equation*}
with the relations
\begin{align*}
x_4x_{26}&=x_8x_{26}=x_{20}x_{26}=0,\\
x_{20}^3&=x_{48}x_4^3+x_{36}x_8^3-x_{20}^2x_8^2x_4.
\end{align*}
Furthermore, 
\begin{enumerate}[(1)]
\item The restriction map
$$\res_T\colon H^*(BF_4;\ff_3) \longrightarrow H^*(BT;\ff_3)^{W(F_4)}$$
factors through
$$\overline\res_T\colon H^*(BF_4;\ff_3)/\sqrt0  \longrightarrow
H^*(BT;\ff_3)^{W(F_4)},$$
mapping the classes 
$x_4$, $x_8$, $ x_{20}$, $ x_{36}$, and $ x_{48}$ to $p_1$, $\bar p_2$,
$\bar p_5$, $\bar p_9$, and $\bar
p_{12}$, respectively.

\item The ideal generated by $x_4$, $x_8$, $x_{20}$ is closed under
the action of the Steenrod reduced power operations, hence, dividing
out by this ideal we are left with a polynomial algebra $\ff_3
[x_{26},x_{36},x_{48}]$ with the following Steenrod algebra action:
\begin{align*}
\pp^1 x_{26}&=0 & \pp^1 x_{36}&=0 & \pp^1 x_{48}&=x_{26}^2\\[-1ex]
\pp^3 x_{26}&=0 & \pp^3 x_{36}&=x_{48} & \pp^3 x_{48}&=0 \\[-1ex]
\pp^9 x_{26}&=x_{36}x_{26} & \pp^9 x_{36}&=-x_{36}^2 &
  \pp^9 x_{48}&=-x_{48}x_{36}
\end{align*}
It turns out that this polynomial algebra is isomorphic, as algebras
over the Steenrod algebra, to $P[u_1,u_2,u_3]^{SL_3(\ff_3)}$ (see
\fullref{exemple2}). Call $\varphi$ the projection 
\begin{equation*}
\varphi \colon H^* (BF_4; \ff_3)/\sqrt 0  \longrightarrow
P[u_1,u_2,u_3]^{SL_3(\ff_3)}.
\end{equation*}
It is a homomorphism of algebras over the Steenrod algebra of reduced
powers with $\varphi (x_{26})=L_3$, $\varphi (x_{36})=Q_{3,2}$, $\varphi
(x_{48})=Q_{3,1}$, using the notation of \fullref{exemple2}.

\item Similarly, if we further divide out by $x_{26}$, 
the quotient $P
[x_{36},x_{48}]$ 
can be identified, as algebras over the Steenrod algebra with the 
subalgebra of
$$P[x_1,x_2]^{GL_2(\ff_3)}=P[Q_{2,1},Q_{2,2}]$$
generated by $Q_{2,1}^3$ and $Q_{2,2}^3$ (see \fullref{exemple1}). We
can then check that 
\begin{multline}\label{Rector-formula}
H^* (BF_4; \ff_3)/\sqrt 0  \cong \\[-1ex]
P[x_{26},x_{36},x_{48}]\prod_{P[x_{36},x_{48}]}
\dfrac{P [x_4,x_8,x_{20},x_{36},x_{48}]}{(x_{20}^3=x_{48}x_4^3
+ x_{36}x_8^3 - x_{20}^2x_8^2x_4) }
\end{multline}
or, in other words, it
fits in the pull-back diagram of algebras over the 
Steenrod algebra of reduced powers:
\begin{equation}\label{Rector-diagram}
\xymatrix{
H^* (BF_4; \ff_3)/\sqrt 0 \ar[r]^\varphi \ar[d]_{\overline\res_T} & P[u_1,u_2,u_3]^{SL_3(\ff_3)} 
                                \ar[d]^{\varsigma} \\ 
  H^* (BT;\ff_3)^{W(F_4)} \ar[r]^{\varrho}& P[x_1,x_2]^{GL_2
(\ff_3)} \rlap{,}}
\end{equation}
where
\begin{align*}
\varrho\circ\overline\res_T( x_{36} ) &= \varrho( \bar p_9)=Q_{2,1}^3 &
\varrho\circ\overline\res_T( x_{48} ) &= \varrho( \bar p_{12})= Q_{2,2}^3
\\[-1ex]
\varsigma(Q_{3,2})&=Q_{2,1}^3 &
\varsigma(Q_{3,1})&=Q_{2,2}^3,
\end{align*}
while other generators are mapped trivially. 
\end{enumerate}

Our aim now is to extend the above diagram to one that captures the
whole structure of $H^* (BF_4; \ff_3)$. The main theoretical tool is
the nil-localization functor for algebras over the Steenrod algebra
(see Broto and Zarati \cite{BroZar} and Schwartz \cite{Schwartz}).

Let $\uu$ be the category of unstable modules over the Steenrod algebra and  
let $\kk$ be the category of unstable algebras over the Steenrod algebra. 
A morphism $f\colon R\to S$ of $\uu$ or  $\kk$ is called a nil-equivalence if
the induced map $\Hom_{\uu}(S,H^*V) \to \Hom_{\uu}(R,H^*V)$ is a bijection
for any elementary abelian $p$--group $V$, and $H^*V=H^*(BV,\F_p)$.
Given an object
$K$ of $\uu$, its nil-localization is another object $\nn_{\uu}^{-1}(K)$ of $\uu$ together with a
nil-equivalence
$\mu_K\colon K\to \nn_{\uu}^{-1}(K)$  which is final among nil-equivalences
with source $K$. If $K$ is an object of $\kk$, then so is
$\nn_{\kk}^{-1}(K)=\nn_{\uu}^{-1}(K)$ and the universal map
$\mu_K$ is a morphism of $\kk$ . We will say that $K$ is reduced
 if $\mu_K$ is injective and nil-closed if $\mu_K$ is an isomorphism. 

The Quillen map for a compact Lie group $G$, expressed as restriction from 
$H^*(BG,\F_p)$ to the inverse limit of cohomologies of elementary abelian 
$p$--subgroups $E$ of $G$ and morphisms induced by conjugation in $G$,
$$q_G\colon H^*(BG;\F_p)\longrightarrow \lim_{E\in\mathcal{E}_p(G)}H^*(BE;\F_p)$$
turns out to be the nil-localization of $H^*(BG,\F_p)$. Thus, the Adams conjecture can be rephrased
by saying that for a compact and connected Lie group $G$ and an odd prime $p$, $H^*(BG;\F_p)$
is a reduced object of $\kk$.

By applying the nil-localization functor to Rector's diagram \eqref{Rector-diagram} we obtain
our main result, that proves the conjecture of Adams for $G=F_4$ and $p=3$.

\begin{thm}\label{main} There is a pull-back 
diagram:
\begin{equation}\label{new-diagram}
\xymatrix{
 \nn_{\kk}^{-1}\bigl(H^*(BF_4;\ff_3)\bigr) \ar[r]\ar[d]&  H^*(BV_3;\ff_3)^{SL_3(\ff_3)} \ar[d] \\
  H^* (BV_4;\ff_3)^{W(F_4)} \ar[r] &  H^*(BV_2; \ff_3)^{GL_2(\ff_3)}
}
\end{equation}
and the nil-localization $\mu\colon H^*(BF_4;\ff_3)\to \nn_{\kk}^{-1}\bigl(H^*(BF_4;\ff_3)\bigr)$ 
is injective.
\end{thm}

The extension of diagram \eqref{Rector-diagram} to \eqref{new-diagram}
requires the fact that the nil-localization of an object $K$ of $\kk$
coincides with that of $K/\sqrt0$. Notice, though, that the natural
projection $K\to K/\sqrt0$ is not in general a morphism of $\kk$. It
might not commute with the action of the Bockstein operator.  Indeed,
in our case, $x_{25}$ is in the radical of $H^*(BF_4;\ff_3)$ but $\beta
x_{25}=x_{26}$ is not nilpotent. In order to overcome this difficulty we
introduce $\kk'$, the full subcategory of objects of $\kk$ concentrated
in even degrees and the right adjoint functor $\tilde{O}\colon \kk \to
\kk'$ of the inclusion functor, described for any $K$ of $\kk$ as the
subalgebra of even degree elements anihilated by the right ideal of
the Steenrod algebra generated by the Bockstein operator (see Broto and
Zarati \cite{BroZar} and Schwartz \cite{Schwartz}).
This adjoint pair provides a natural map $j\colon \tilde O K\to K$,
and the composition
\begin{equation*}
\kappa_K\colon \tilde{O} K \to K \to K/\sqrt 0
\end{equation*}
is clearly a morphism of $\kk'$. Moreover, it is a nil-equivalence.
In fact,  $j\colon \tilde O K\to K$ is always injective and a
nil-equivalence, so the kernel of $\kappa_K$ is the radical of $\tilde O
K$. An element in the cokernel of $\kappa_K$ is represented by an element
of $K$, but the $p$th power of any element of $K$ belongs to $\tilde OK$,
hence this cokernel is also nilpotent, so what we obtain is a diagram
of nil-equivalences
$$\bfig\barrsquare<800,400>[\tilde O K` K` K{/}\sqrt0` \nn_{\kk}^{-1}K;
  j` \kappa_K` \mu_K` \mu_{K{/}\sqrt0}]\efig$$
Notice that a nil-equivalence between objects of $\kk'$ is precisely
an (F)--isomor\-phism in the sense of Quillen \cite{Quillen}.  If $K$
is a nil-closed object, then  $\kappa_K$ is an isomorphim and $j\colon
K/\sqrt0\cong\tilde OK\to K$ is the nil-localization.

\begin{proof}[Proof of \fullref{main}]
For an elementary abelian $p$--group $V$ and $G$ a subgroup of $GL(V)$,
$H^*(BV,\F_p)$ is nil-closed and then
$$S(V^*)\cong H^*(BV,\F_p)/\sqrt0\cong\tilde O H^*(BV,\F_p).$$
Since $\tilde O$ commutes with inverse limits and the inverse limit
of nil-closed objects is nil-closed, we also have that $S(V^*)^G\cong
\tilde O H^*(BV,\F_p)^ G$ and the inclusion $S(V^*)^G \to H^*(BV,\F_p)^
G$ is the nil-localization of $S(V^*)^G$.

Similarly, the  inverse limit of a functor $c\in\cc\mapsto
H^*(V_c;\F_p)^{G_c}\in \kk$ is  nil-closed and $\tilde O \lim_{c\in\cc}
H^*(V_c;\F_p)^{G_c} = \lim_{c\in\cc} S(V_c^*)^{G_c}$, hence if $L=
\lim_{c\in\cc} S(V_c^*)^{G_c}$, then $\nn_{\kk}^{-1}L= \lim_{c\in\cc}
H^*(V_c;\F_p)^{G_c} $.  This applies to the pull-back diagram
\eqref{Rector-diagram} and proves that \eqref{new-diagram} in the
statement of the theorem is again a pull-back diagram.

We will identify the composition of $\mu\colon H^*(BF_4;\F_3) \to
\nn_{\kk}^{-1} \bigl(  H^*(BF_4;\F_3)  \bigr)$ with each of the maps in
diagram \eqref{new-diagram} to the cohomology of an elementary abelian
$3$--subgroup.

\begin{enumerate}[(1)]
\item $H^*(BF_4;\ff_3) \to H^*(BV_4;\ff_3)^{W(F_4)}$. This
map clearly factors as
\begin{equation*} H^*(BF_4;\ff_3)
\stackrel{\res_T}{\to} H^*(BT;\ff_3)^{W(F_4)} \to
H^*(BV_4;\ff_3)^{W(F_4)}
\end{equation*}
and the kernel is the ideal of $H^*(BF_4;\ff_3)$ generated by $x_9$,
$x_{21}$, $x_{25}$, $x_{26}$.

\item $\widehat\varphi\colon H^*(BF_4;\ff_3)\to
H^*(BV_3;\ff_3)^{SL_3(\ff_3)}$ is defined as the composition
\begin{equation*}
H^*(BF_4;\ff_3)\to \nn_{\kk}^{-1}(H^*(BF_4;\ff_3))\to
H^*(BV_3;\ff_3)^{SL_3(\ff_3)}
\end{equation*} obtained by applying the nil-localization functor to 
$\varphi$: 
\begin{equation*}
\xymatrix{ \tilde{O} H^*(BF_4;\ff_3) \ar[r] \ar[d] &
H^*(BV_3;\ff_3)/\sqrt0
\ar[r]^-\varphi \ar[d] & S(V_3^*)^{SL_3(\ff_3)}\ar[d]\\
 H^*(BF_4;\ff_3) \ar[r]& \nn_{\kk}^{-1}
 H^*(BF_4;\ff_3) \ar[r]& H^*(BV_3;\ff_3)^{SL_3(\ff_3)} }
\end{equation*}

Notice that $p$th powers of even diamensional elements in an
object $K$ of $\kk$ belong to $\tilde O(K)$. In particular, 
$x_{26}^3, x_{36}^ 3,x_{48}^3$ belong to $\tilde{O} H^*(BF_4;\ff_3) $, thus 
they are mapped to $L_3^ 3, Q_{3,2}^ 3, Q_{3,1}^ 3\in H^*(BV_3;\ff_3)^{SL_3(\ff_3)} $, 
respectively by $\widehat\varphi$.

Now, the other generators of $H^*(BF_4;\ff_3)$, 
 $x_{25}, x_{21}, x_{20}, x_{9}, x_{8}, x_4$ are linked by Steenrod operations
to $x_{26}$ (see \fullref{TodaThm}),  hence they can not be in the kernel of 
$\widehat\varphi$. 

The invariant ring  $H^*(BV_3;\ff_p)^{SL_3(\ff_p)}$ is described in \fullref{exemple2}. If $p=3$,
besides the polynomial generators $L_3$, $ Q_{3,2}$ and $Q_{3,1}$, we have
$m_3=M_{1,2,3}$, $m_4=\beta m_3 = M_{2,3}$, $m_8=-\pp^1m_4 = M_{1,3}$, 
$m_{20}=\pp^p m_8 = M_{1,2}$, $m_9 = \beta m_8 = M_3$, $m_{21}=\pp^pm_9 = M_2$, 
and $m_{25}=\pp^1m_{21} = M_1$. Recall also that $\beta m_{25} = L_3$. Here the 
subindices of the lowercase $m$'s indicate the degree in which they appear. 

It follows that $\widehat\varphi(x_4)$ can only be $\pm m_4$ and since $\widehat\varphi(x_{26}^3)=
L_3^ 3$, it has to be $+m_4$, and
$$\widehat\varphi \colon H^* (BF_4;\ff_23) \longrightarrow
  H^*(BV_3;\ff_3)^{SL_3(\ff_3)}$$
maps
\begin{align*}
x_4 &\mapsto m_4 & x_8 &\mapsto m_8 & x_9 &\mapsto m_9\\
x_{20} &\mapsto m_{20} & x_{21} &\mapsto m_{21} & x_{25} &\mapsto m_{25}\\
x_{26} &\mapsto L_3 & x_{36} &\mapsto Q_{3,2} & x_{48} &\mapsto Q_{3,1}
\end{align*}
\end{enumerate}

It is now routine to check that 
\begin{equation*}
\ker \widehat\varphi =(x_4^2,x_8^2, x_{20}^2, x_{20}x_8, x_{20}x_4, x_8x_4 )
\end{equation*}
and that this ideal is contained in the subalgebra of 
$H^*(BF_4;\ff_3)$ generated by 
$x_4,\allowbreak x_8,\allowbreak x_{20},\allowbreak x_{36},\allowbreak x_{48}$ which is detected in  
$H^*(BT;\ff_3)^{W(F_4)}$, hence the composition 
\begin{equation*}
H^*(BF_4;\ff_3) \To \nn_{\kk}^{-1}\bigl(H^*(BF_4;\ff_3)\bigr) \To
H^*(BT;\ff_3)^{W(F_4)} \times 
H^*(BV_3;\ff_3)^{SL_3(\ff_3)}
\end{equation*}
given by $\res_T$ and $\widehat\varphi$ is injective.
\end{proof}

The map $\widehat\varsigma\colon H^*(BV_3;\ff_3)^{SL_3(\ff_3)} \to H^*(BV_2;\ff_3)^{GL_2(\ff_3)}$, obtained as extension of 
$\varsigma $ in diagram \eqref{Rector-diagram}, 
maps $m_3$ trivially (by degree reasons) and therefore $m_4,\allowbreak
m_8,\allowbreak m_9,\allowbreak
m_{29}, \allowbreak m_{21}, m_{25}, L_3$ are mapped trivially, too, and
the image of $\widehat\varsigma$ is $P[x_{36},x_{48}]$, which coincides
with the image of $\varsigma$. We can therefore express the mod $3$
cohomology of $BF_4$ as the pull-back diagram
\begin{equation*}
\xymatrix{H^* (BF_4;\ff_3) \ar[d]_{\res_T}\ar[r]^-{\widehat\varphi} & \im
\widehat\varphi \ar[d]^{\widehat\varsigma\big|_{\im \widehat\varphi }} & \hspace{-1.5cm}\subset H^*(BV_3;\ff_3)^{SL_3(\ff_3)} \\
H^* (BT;\ff_3)^{W(F_4)} \ar[r]^{\qua\qua\varrho} & P[x_{36},x_{48}]  &}
\end{equation*}
where $\im \widehat\varphi $ is the subalgebra of 
$H^*(BV_3;\ff_3)^{SL_3(\ff_3)} $ generated by 
$m_4$, $m_8$, $ m_9$, $ m_{29}$, $ m_{21}$, $
m_{25}$, $ L_3$, $ Q_{3,2}$, $ Q_{3,1}$, thus leaving in the cokernel only
$\coker  \widehat\varphi \cong m_3P[Q_{3,2},  Q_{3,1}]$, or,  in other words,
$$H^* (BF_4;\ff_3)  \cong \im \widehat\varphi\prod_{ P[x_{36},x_{48}]  }
H^* (BT;\ff_3)^{W(F_4)} \,,  $$
which we think  is the correct way to understand this cohomology ring as algebra over the Steenrod
algebra.

\bibliographystyle{gtart}
\bibliography{link}

\end{document}